\documentclass[10pt]{amsart}
\usepackage{amssymb,amsfonts}
\usepackage{xypic}
\usepackage{amsthm,amsmath}

\newtheorem{thm}{Theorem}[section]
\newtheorem{cor}[thm]{Corollary}
\newtheorem{pro}[thm]{Proposition}
\newtheorem{lem}[thm]{Lemma}

\theoremstyle{definition}
\newtheorem{defn}[thm]{Definition}

\newtheorem{exmp}[thm]{Example}

\newcommand{\zz}{\mathbb{Z}}
\newcommand{\rr}{\mathbb{R}}

\newcommand{\qq}{\mathbb{Q}}

\newcommand{\spec}{\text{spec} \, }
\newcommand{\rarr}{\rightarrow}

\title{ Survey on the Burnside ring of compact Lie groups  }
\author{ Halvard Fausk}  \email{fausk@math.uio.no}
\thanks{The author thanks  Marius Thaule for helpful comments.}
\address{Department of Mathematics,
University of Oslo,
1053 Blindern,
0316 Oslo,
Norway}

\begin{document}

\begin{abstract}
The definition and basic properties of the
Burnside ring of    compact Lie groups are presented, with emphasis  on the analogy with the construction of the   Burnside ring of   finite groups.
\end{abstract}\maketitle
The Burnside ring of a finite group encodes the ``calculus of cosets'' of the  group.  It was defined by Burnside in his work on tables of marks of finite groups \cite[page 236]{bur}.
 The Burnside ring of a   compact Lie group was defined  by tom Dieck in the context of  equivariant stable homotopy theory.
 It can also be described as  encoding  the  ``calculus of cosets'', provided  only certain transitive orbits of $G$  are made
 visible. Namely, those transitive orbits, $G/H$, such that $H$ has finite order in its normalizer.
Some  references
are \cite{max}  \cite{bside1} \cite{tdi} \cite{td} \cite{dre} \cite[Chapter V]{lms}.

\section{The Burnside ring of  a finite group}
Let $G$ be a finite group. The Burnside ring of $G$  is the Grothendieck
group completion of the semiring of isomorphism classes of
finite $G$-sets. It is  denoted by $A(G)$ (other notations are   $B(G)$ and  $\Omega (G)$). Addition is given by disjoint union, and
multiplication  by Cartesian product. These operations are  well defined on isomorphism classed  of $G$-sets.
 The Burnside ring
   of $G$  is isomorphic, as an abelian group,   to  the free abelian group with generators  the  isomorphism
  classes of transitive $G$-sets.
Under this identification, the
  multiplication of the additive generators  is given  by the double coset formula. The double coset formula says that  the $G$-set  $G /H \times G / K$  is $G$-isomorphic to the  disjoint union of the transitive $G$-sets $ G / (H \cap  g K g^{-1} )$ where $HgK$ runs over the double coset $ H \,\backslash  G /K$.

Let $H$ be a subgroup of $G$ and let $X$ be a $G$-set. The
$H$-fixed point set  $X^H $ is the subset $\{ x \in X \, | \, hx = x , h\in H \}$ of $X$.
The number of elements in $X^H$, denoted $|X^H|$, only depends on the $G$-isomorphism class of $X$ and  the $G$-conjugacy class of $H$.
For every conjugacy class of a subgroup $H$ of $G$ the map $$  X \mapsto |X^H | $$  gives a
semiring homomorphism from the semiring of isomorphism classes of
finite $G$-sets to the integers. Hence there is an induced $H$-fixed point ring homomorphism $ \phi_H  \colon  A(G) \rarr \zz$.  These
$H$-fixed point ring homomorphisms ensemble to give a
ring homomorphism $$ \phi  \colon  A(G) \rarr \textstyle\prod_{(H)}
\zz , $$ where the product is over the $G$-conjugacy classes, $(H)$, of subgroups $H$ of $G$.

The map $\phi$ is sometimes called the mark homomorphism. Choose a linear ordering  of the conjugacy classes of subgroups of $G$ that respects subconjugacy.
The matrix   with the $(H)$,$(K)$ entry given by  $\phi_K ( [G / H])$ is called the table of marks, or the mark matrix, of $G$.

The basic properties of the Burnside ring of a finite group are described
nicely in the first chapter of \cite{td} and in \cite{dre}. A  recent survey is \cite{bouc}.

\section{Recollections about  compact Lie groups} \label{compactLie}
Compact Lie groups need not be connected, so all finite groups are compact Lie groups. Only closed subgroups of compact Lie groups are considered. The Weyl group $W_G H$ of a subgroup $H$ in
$G$ is  $ N_G H / H $.

 A theorem of Montgomery and Zippin says that for any closed
subgroup $H$ of $G$ there is an open neighborhood $U$ of the
identity element in $G$ such that all subgroups of $H U $ are
subconjugate to $H$ \cite[II.5.6]{bre} \cite{maz}.

Let $H $ and $  K $ be subgroups of $G$. The normalizer $N_G H$
acts from the left on $(G/K)^H$.  Montgomery and Zippin's
theorem implies that the coset $(G/K)^H /N_G H $ is finite
\cite[II.5.7]{bre}.  In particular,  if $W_G H$ is finite, then
$(G/K)^H$ is finite. The Weyl group $W_G K$ acts freely on $
(G/K )^H$ from the right by $gK \cdot nK = gnK$, where $gK \in
(G/K)^H$ and $nK \in W_G K$.  So $ |W_G K|$ divides $| (G/K
)^H|$. The fixed point space $(G/K)^H$ is nonempty if and only if $H $ is subconjugated to $K$ in $G$. Hence, if   $H $ is subconjugated to $ K $ and $H$ has finite Weyl group, then $K$ also has finite Weyl group.

Let $J$ be a subgroup of $G$ and let $X$ be a $G$-space.  A point $x $ in  $  X$ has orbit type $J$ if the stabilizer
 $ \{ g \in G \ | \ g x = x \} $  of $x$ is conjugate to $J$.
 Let $X_{(J)} $ denote
the subspace of  $X$ consisting of the points of $X$ with orbit type
$J$.  Let   $X_{\text{fin}}$
denote the subspace  of  $X$ consisting of all points whose
 stabilizers have
finite Weyl group. If $H$ has finite Weyl group and $X$ is a $G$-space, then $X_{\text{fin}}^H = X^H$ by the results in the previous paragraph.
A $G$-CW-complex is a space built out of  $n$-dimensional  $G$-cells $ S^{n-1}\times G/H \to D^n \times G/H$, for subgroups $H$ of $G$, by  gluing them to cells of dimension $(n-1)$ or lower, for $n \geq 0$  (see \cite[I.3]{ala}  for a precise definition).
 If $X$ is a $G$-CW-complex, then $X_{\text{fin}} $ is a subcomplex of $X$. This follows because the stabilizer of any point
 in $D^n \times G/H$ is  conjugated to $H$, and   cells whose stabilizers have finite Weyl group can only map to other cells whose stabilizers also have finite Weyl group.

\section{Construction of the Burnside ring of  a compact
Lie group}
The basic idea in the definition of the Burnside ring of a compact Lie group $G$ is to consider  finite disjoint unions of $G$-orbits, but  ignore those orbits $G/H$  where  $H$ does not have finite Weyl group in $G$.
 Denote the  $G$-isomorphism class of a $G$-space $X$ by $[X]$.
The set of $G$-isomorphism classes of finite disjoint unions of transitive $G$-spaces, whose orbit types have finite Weyl group,
 has a structure of an abelian  semigroup given by disjoint union.

If $G$ is a compact Lie group, then the Cartesian product of two homogeneous $G$-spaces $G/H$ and $H/K$
is not isomorphic to a disjoint union of homogeneous $G$-spaces. So
  the definition of the product has  to be modified.  If $G$ is a finite group, then a reformulation of  the double coset formula says that
   $$ [G/H ] \cdot
   [G/K] = \Sigma_{(J)} |(G/H \times G/K )_{(J)}/G | \,  [G/J].$$  The following  key observation is due to Schw\"anzl \cite{shw}.

\begin{lem} \label{sch}  Assume that $Z$ is a $G$-space such that $Z^J$ is a finite subspace of $Z$.  Then $ Z_{(J)}/ G$  is  finite. \end{lem}
\begin{proof}  In fact there  is an inequality   $| Z_{(J)}/ G| \leq | Z^J|$. It suffices to check this for $G/gJg^{-1}$.  In this case the claim is true since   $ ( G/ gJg^{-1})/G $ is a point and  $(G/ gJg^{-1} )^J$ is nonempty,  for all $g \in G$.   \end{proof}

In particular, if $X$ and $Y$ are $G$-spaces so
that $ X^J$ and $Y^J$ are finite, then   $ (X \times Y)_{(J)} / G$ is finite. Hence, if $J$ is a subgroup of $G$ with finite Weyl group, and $H$ and $K$ are subgroups of $G$, then $G/H^J$ and $G/K^J$ are finite, and therefore $ (G/ H \times G/K)_{(J)} / G $ is finite.

  Illman has proved that  the product $G/H
\times G/K$ is a finite $G$-CW-complex \cite{ill}.
The $G$-cells of $G/H \times G/K$ with stabilizer a subgroup of $G$ with finite Weyl group are all 0-dimensional by Lemma \ref{sch}. Hence $(G/H \times G/K)_{\text{fin}}$, the subspace of $G/H \times G/K$ obtained by removing the   $G$-cells whose  stabilizers   do not have finite
 Weyl group, is a finite disjoint union of homogeneous
 $G$-spaces.

\begin{defn}
\label{defn:mult}
Define  a product as follows $$ [ G/ H ] \cdot [ G/K ] = [ (G/ H \times  G/K)_{\text{fin}} ] = \Sigma_J \, n_J [G / J],
$$ where the sum is over the
conjugacy classes of subgroups $J$ of $G$
 with finite Weyl group,
 and $n_J $ is the number of elements
 in the finite set $ (G/ H \times  G/K)_{(J)} / G  $
 \cite{shw}.
 \end{defn}
The sum is finite since $ G/H \times G/K$ has only finitely
many orbit types. The isomorphism class of the point, $[G/G]$, is the
multiplicative unit. The multiplication is clearly commutative
and distributive with respect to the addition.
\begin{lem} The   multiplication in Definition \ref{defn:mult}
is associative. \end{lem}
\begin{proof}  Consider three subgroups   $ H , J $, and $ K $ of $G$ all with finite Weyl groups.
It suffices to show  that $ ((G/H \times G / J)_{\text{fin}} \times G/ K )_{\text{fin}}  $
is equivalent to $ (G/H \times G / J \times G/ K )_{\text{fin}}  $.
Let  $U$ be  a subgroup of $G$ with
infinite Weyl group.  Then  $ G/U \times G/K$
consists of $G$-cells with stabilizers that are subconjugated to $U$. Hence they all have  infinite Weyl
groups (see section \ref{compactLie}).  The result follows.   \end{proof}

\begin{defn} \label{defn:Burnside} Let $G$ be a compact Lie group.
Then the  Burnside ring $A(G)$   is the
 Grothendick construction of the  semiring of isomorphism classes
 of finite disjoint
 unions of homogeneous  $G$-spaces, $G/H$, for closed subgroups $H$ of $G$ with finite
 Weyl group; the sum is given by disjoint union and the multiplication
 is given by Definition \ref{defn:mult}.
 \end{defn}

Recall from section \ref{compactLie} that  $(G/K)^H$  is finite, whenever   $H$ is a subgroup of  $G$ with finite Weyl group.
\begin{lem} \label{fixedpthom} Let $H$ be a subgroup of $G$ with finite Weyl group.
The function $\phi_H  \colon  A (G)  \rarr \zz$, defined  by
$\phi_H (G/K) = | (G /K)^H | $ on the free generators $[G/K]$ of $A(G)$,  is a ring homomorphism. \end{lem} \begin{proof}
It suffices to show that $\phi_H$ is a semiring homomorphism before passing  to the Grothendieck construction.
The map is well defined, additive and  respects both the additive and the
multiplicative units. The map  respects the
multiplication  by observing that
$$ \begin{small}   (G / K \times G/ L)_{\text{fin}}^H =  (G / K \times G/ L)^H \end{small} $$ since $H$ has finite Weyl group.
\end{proof}

\section{Other definitions of  Burnside rings} \label{alternatives}
There is a general categorical approach to Burnside rings
\cite{may}. Peter May associates to any symmetric tensor triangulated category (with a skeletally small category of dualizable objects) a Burnside ring.  It is a subring of the ring of selfmaps of the unit object.  The Burnside ring of a compact Lie group is  isomorphic to the Burnside ring associated to the
$G$-equivariant stable homotopy category   (see section \ref{sec:htptheory}).

The following is a  sketch of the main ideas  in tom Dieck's original construction of $A(G)$ for compact Lie groups.  The general categorical definition of May is modeled on this example.
Instead of  disjoint unions of homogeneous  $G$-spaces
 the full subcategory of the stable homotopy category  consisting of the  dualizable objects is used. A spectrum is dualizable if and only if it is a  suspension spectrum, $  \Sigma^{\infty}_G  X$,  of a retract of finite $G$-cell complex $X$. There is a semiring of stable $G$-homotopy classes of such objects.
 (Alternatively,  start out with  all $G$-spaces of the homotopy type of a retract of a finite
$G$-CW-complex. There is a semiring of $G$-homotopy types  of
such spaces; the  sum is given by disjoint union, and the
 product  by  Cartesian
product.) There is a semiring homomorphism into the integers given by Euler characteristic of the (geometric) $H$-fixed point spectra,
for each subgroup $H$ of  $G$.  These maps induces ring homomorphisms from the
 Grothendieck construction of the semiring of stable $G$-homotopy classes of dualizable  objects to the integers.   The Burnside ring is the
  ring  obtained by dividing  out by the intersection of the  kernel ideals of  these homomorphisms for all subgroups $H $ of $G$.  Hence a formal difference of two
  stable $G$-homotopy types  $ \Sigma^{\infty}_G X$ and $ \Sigma^{\infty}_G  Y$ is equal to 0 in the Burnside ring
if and only if the Euler characteristic of the fixed point spaces  $ X^H $ and $ Y^H $
are equal for all closed subgroups $H$ in $G$.
 It is enough
to check this for all subgroups $H$ with finite Weyl groups since
$\chi (X^K) = \chi ( X^H) $, whenever  $ H / K $ is a torus.
The comparison between  the definition of the Burnside ring sketched above and the one  given in Definition  \ref{defn:Burnside} uses the
additivity of the Euler characteristic on stable  cofiber sequences.

The Burnside ring from the perspective of stable equivariant homotopy theory and geometric topology are surveyed in  \cite{ala} \cite{mas}.

\section{Maps between Burnside rings}
Let $ G_1$ and $G_2$ be two compact Lie  groups.
Then there is a natural map $$ p \colon   A ( G_1 )
\otimes A( G_2 ) \rarr
A ( G_1 \times G_2 )  $$
 given by sending $[G_1 / H_1] \otimes [G_2 / H_2 ]$ to $[ G_1 \times G_2 / H_1 \times H_2 ]$.
 The map $p$ is an injective ring map, however it is not an  isomorphisms unless all subgroups of $G_1 \times G_2$, with finite Weyl groups,  are of the form $H_1 \times H_2$, for subgroups $H_1 \leq  G_1$ and $H_2 \leq G_2$. If $G_1 $ and $G_2 $ are finite groups and
$|G_1|$ and $|G_2|$ are relative prime, then $p$  is an
isomorphism.

Let $G$ be a finite group. Then there is a map
 $$ \alpha  \colon A ( \zz / |G|) \rarr A(G) $$ such that  for any subgroup $H$ of $G$ the composite map
 $\phi_H \circ \alpha   $ is equivalent to $\phi_{ \zz / |H| }$ where $\zz / |H|$ is the (unique) cyclic subgroup of $\zz / |G|$ of order $|H|$ \cite{app}.

Let $H$ be a subgroup of a compact Lie group $G$.
The induction map  $\text{ind} \colon A(H) \rarr A(G)$
is given by sending a generator $ [ H /K ] $ to $[ G/ K ] $ if $K$ has finite Weyl group in $G$ and to 0 otherwise.

The restriction map between Burnside rings is most easily described
when the Burnside ring is defined in terms of equivalence  classes of compact $G$-spaces
 (see section \ref{alternatives}).
The restriction map $\text{res} \colon A(G) \rarr A(H)$ is    defined   by sending the isomorphism class of a $G$-space $ X$ to $ X|H$, $X$  regarded as an $H$-space via the inclusion $ H < G$.
Let $L$ be a subgroup of $H$ with finite Weyl group in $H$, then $$ \phi_L \text{res} \,  x = \phi_{L'} x , $$ where $L' $  is an extension in $G$ of $L$ by a torus such that $L'$ has
finite Weyl group in $G$. 
Let $H$ be a normal subgroup of a finite  group $G$. The
restriction map $ A(G) \rarr A(H)$ is give by sending the
isomorphism class of a
$G$-set $G/ L$ to  $$ \frac{|G| |H \cap L|}{|H|  |L|} \,
[ H / H \cap L ] , $$  for any subgroup $L$ of $G$.

The restriction map $A(G) \rarr A( {1} )\cong \zz$ is called the augmentation map. The kernel of this map is called the augmentation ideal. The augmentation map is given by sending $ [G / L] $ to $ | (G/L)^T|$ where $T$ is a maximal torus in $G$.

\section{Examples of Burnside rings of abelian groups}
The only compact Lie groups with no proper subgroups with finite Weyl groups are the trivial group and the tori:   $$ A(1) \cong \zz  $$
$$ A((S^1)^n) \cong \zz ,  $$ for any $n \geq 1$.  If $G$ is a  compact abelian  Lie group, then $G$ is isomorphic to the cartesian product of a torus and a finite abelian group.
Hence if $G$ is a compact  abelian Lie group, then $ A(G) \cong A ( G / G^{\circ})$, where $G^{\circ}$ is the unit component of $G$ and $G / G^{\circ}$ is the group of components of $G$.
A finite abelian group is isomorphic to the cartesian product of $p$-groups. Hence to calculate the Burnside ring of  compact abelian Lie groups it suffices to calculate
 the Burnside ring of  finite abelian
$p$-groups, for primes $p$. They are  of the form
$G = \zz / p^{n_1} \times
\cdots \times \zz / p^{n_m}$ where   $ m \geq 1$ and $ n_i
\geq 1$.

  The multiplication
in the Burnside ring (double coset formula) is particularly
simple for finite abelian groups. It is
 given by  $$  [G / K] \cdot [G / L]  \cong \frac{|G| \, |K \cap L
|}{ |K| \, |L|} \,  [G / K \cap L] ,$$ for subgroups $K$ and $L$ of $G$.
 While it is easy to
find the isomorphism classes of   subgroups of $G$,
 it is more  involved to  keep track of
  all subgroups of $G$, and their intersections.
In the rest of this section
bookkeeping of the subgroups is described in the special case of cyclic $p$-groups $\zz / p^n$ and elementary $p$-groups $(\zz / p)^n$.

  There is an  isomorphisms
 $$A(\zz / p ) \cong \zz [x] / (x^2 - p x) . $$  More generally,   $$ A(\zz / p^n) \cong \zz  [a_1 , \ldots , a_n]
/ a_i a_j  = p^i a_j \ \  (\text{for }  j \geq i ), $$  where $a_i $ denotes
the coset of $\zz / p^n$ by the subgroup  $\zz / p^{n-i}$, for $i= 1, \ldots n $.

 The Burnside ring $A((\zz / p)^2)$ is isomorphic
to
 $$  \zz  [ a_0 , \ldots , a_{p} , b ] /
a_i b= p b, b^2 = p^2 b , a_{i}^2 = p a_i , a_i a_j = b \text{ for }
  i \not= j ,$$
where  $b$ is the coset $(\zz / p)^2$ and $a_i$ is the coset of $(\zz /
p)^2$ by the subgroup  generated by the element $ (1 , i)$ for $  0 \leq i < p$
and the subgroup generated by $ (0 , 1)$ when $ i = p$.

To bookkeep  the subgroups of $(\zz / p)^n$ and their intersections the  tactic  is  to
 associate to any  subgroup  $H$ of $(\zz / p)^n$ a
 distinguished
set of generators of $H$.
This gives a
systematic description of the subgroups of $(\zz / p)^n$ and their intersections.

Fix $n$ and let $H$ be a subgroup of $(\zz / p)^n$.  There is a tuple of integers (with $m \leq n$)
$$ n \geq i_1 > i_2 > \cdots > i_m \geq 1  $$ and elements
$\alpha_1 , \ldots , \alpha_m \in (\zz / p)^n$ such  that
$\alpha_{k}^{i_k} = 1$,
$ \alpha_{k}^{j} = 0 $ for $j > i_k $,
 and $ \alpha_{k}^{i_l} = 0 $ whenever  $ l
\not= k$. The superscript $j$ denotes the $j$th coordinate in $(\zz / p)^n$.
 The subgroup $H$ is the subgroup of $(\zz / p)^n$ generated by
the elements $\alpha_1 , \ldots , \alpha_m$.
The elements are linearly independent and $H$
 is isomorphic to $(\zz / p )^m$.
 There is exactly one such set of distinguished  generators that
  generates  $H$.
It is illustrative to write the generators  in the form of an $m
\times n$-matrix with values in $\zz / p $.
  The following is an example when the rank of the subgroup is
$m = 3$
$$ \left(
  \begin{array}{cccccccccccccc}
    0 & \cdots & 0 & 1 & * & \cdots & * & 0 & * & \cdots & * & 0 & *&
     \cdots   \\
    0 & \cdots & \cdots & \cdots & \cdots & \cdots & 0 & 1 & * &
    \cdots & * & 0 & * & \cdots
      \\ 0 & \cdots & \cdots & \cdots & \cdots & \cdots & \cdots &
      \cdots & \cdots  & \cdots & 0 & 1 & * & \cdots
  \end{array}
\right)  $$

The intersection of a subgroup given by
$$ n \geq i_1 > i_2 > \cdots > i_m \geq 1  $$ and elements
$\alpha_1 , \ldots , \alpha_m \in (\zz / p)^n$ by  another
subgroup given by $$ n \geq i'_1 > i'_2 > \cdots > i'_{m'} \geq 1  $$
and elements
$\alpha'_1 , \ldots , \alpha'_{m'} \in (\zz / p)^n$ is the
subgroup given  by  the generators
$$ n \geq j_1 > j_2 > \cdots > j_s \geq 1  $$ and elements
$\beta_1 , \ldots , \beta_s \in (\zz / p)^n$ such  that
the set of tuples   $ \{ (j_k , \beta_k ) \}_{k=1 , \ldots , s} $ equals the intersection of the sets of tuples
$ \{ (i_k , \alpha_k ) \}_{k=1, \ldots , m}$  and
$ \{ (i'_k , \alpha'_k ) \}_{k=1, \ldots ,m'}$.

\section{Examples of Burnside rings of  nonabelian groups}
If $H$ and $K$ are subgroups of $G$ with finite Weyl groups, and $H$ is normal in $G$, then
$$  [ G/H ] \cdot [G/K] = \frac{|G/H| \, | (G/K)^{K \cap H} |}{|W_G (K \cap H) |} \, [ G / K \cap H ] , $$ if $ K \cap H$   has finite Weyl group in $G$, and
$ [ G/H ] \cdot [G/K] = 0 $  if $ K \cap H$ does not have finite Weyl group in $G$.

Let $G$ be the permutation group $\Sigma_3$.  It is isomorphic to the semidirect product $ \zz / 3 \rtimes \zz / 2$.
The isomorphism classes of transitive $G$-sets are: $[G /G ]$, $ a = [ G / \, \zz /2 ]$, $b = [ G / \, \zz / 3 ]$, and $c = [G / 1 ]$.
  The Burnside ring $A(G)$ is isomorphic to the polynomial ring $\zz [ a , b, c ] /  \thicksim$, where the relations are
   $ a^2 = a +c$, $b^2 = 2 b$, $ c^2 = 6 c$, $ab = c$, $ a c = 3 c$, and $bc = 2 c$.

Let $G$ be the nontrivial semidirect product $S^1 \rtimes \zz / 2$ (also known as $\text{O}(2)$).  The subgroups of $G$ with finite Weyl groups are $ G$, $S^1 \rtimes 0$, and $ \zz /n \rtimes \zz / 2$ for $n \geq 1$.  The normalizers are $ G$, $G$, and $   \zz /2n \rtimes \zz / 2$ for $n \geq 1$, respectively. Let  $y$ denote the element $ [G /\, S^1 \rtimes 0]$, and let $x_n$ denote $[ G /( \zz /n \rtimes \zz / 2)]$, for $n \geq 1$.
Let $\sim$ be the relations generated by  $ y \cdot y = 2 y$, $ x_n \cdot x_m = 2x_{(n,m)}$, for $m,n \geq 1$, and   $ x_n \cdot y = 0$, for $n \geq 1$, where $(m,n) $ is the greatest common divisor of $m$ and $n$. Then there is an isomorphism
$$ A (S^1 \rtimes \zz / 2 ) \cong \zz [ y ,x_1 ,x_2 ,x_3 , ...]\,  / \sim . $$
The Burnside ring  of $ \text{SU}(3)$ is described in great detail in \cite[5.14]{td}.

Computer programs facilitate the calculation of the table of marks and the Burnside ring for many  groups.
See for example \cite{pfei}.

\section{The mark homomorphism $ \phi : \mathbf{A(G)} \rightarrow \mathbf{C (G) }$} \label{sect:mark}
 Let $\mathcal{C} G $ 
be the space of closed subgroups of $G$ with the Hausdorff
topology from $G$. This is a compact metric space. Let $ \Psi G$ be
the quotient space of $\mathcal{C} $ obtained by identifying
$G$-conjugate subgroups  of $G$. The space  $ \Psi G$ is  countable, hence it is a totally disconnected space \cite{tdf}.

Let $ \Phi G$ be the
subspace of $ \Psi G$ consisting of conjugacy classes of
closed subgroups of $G$ with finite Weyl group.
By Montgomery and   Zippin's theorem the complement $\Psi G - \Phi G $ is open.
So  $\Phi G$ is a closed subspace of $\Psi G$, hence a compact space.

There is a continuous retract map $ \omega  \colon  \Psi G
\rarr \Phi G $     given by sending $ (H) $ to the conjugacy class
of a largest possible  extension of $H$ by a torus \cite[1.2]{fol}. This extension
is unique up to
conjugation. The following result gives a useful description of $\omega$ \cite[2.2]{fol}.
\begin{lem} The conjugacy class
$\omega (H) $ is the conjugacy class of the subgroup generated
by $(H)$ and a maximal torus in the centralizer $ C_G H$.
\end{lem}
 Let $C(G)$ denote the ring of continuous functions from $\Phi
G$ to the integers $\zz$.  Recall Lemma \ref{fixedpthom}.
\begin{lem} The homomorphisms $\phi_H$, for $H \leq G$,
 ensemble to give a ring
  homomorphism
$$ \phi  \colon  A(G) \rarr C(G) .$$
\end{lem}
\begin{proof}
It suffices to show  that the map $(H) \mapsto |(G/K)^H |$ is
continuous for each $ (K)$ in $\Phi G$. Let $ (H_i)$ be a sequence converging to $H$.
Montgomery and Zippin's theorem implies that there is no loss of generality in assuming that  $ H_i
<H$. As an $H$-manifold,  $(G/K)|H$,  has finitely many isotropy  types \cite{ill}.
  The fixed point space  $(G /K)^{H} $ is equal to $  (G /K)^{H_i}$ whenever $H_i$ is not subconjugated, in $H$,
  to any of the isotropy groups of  $(G/K)|H$   that are properly contained in $H$.

Since  $(G/K)|H$ has finitely many orbit types
there is an $\epsilon >0$ such that the distance, in the metric space $\Psi G$,
between $(H)$ and each of the $G$-conjugacy classes of
the  $H$-isotropy groups of  $(G/K)|H$, different from $(H)$, are all greater or equal to $\epsilon$.
Hence if the distance between $(H)$ and $(H_i)$ is less than $\epsilon$, then $|(G /K)^{H}| = | (G /K)^{H_i}|$.
\end{proof}

\begin{lem} \label{lem:phiinj}
The mark homomorphism  $\phi$ is an injective
ring map. \end{lem}
\begin{proof}
Assume that
$$ \textstyle\sum_{ i=1}^{n} q_i \, \phi ([G/H_i ])  =0  $$ where all $ q_i \in \zz$ (or in $\qq $)
are nonzero. Let $ (H_m )$ be a maximal conjugacy class among
 $ (H_1  ), \ldots , (H_n )  $. The function evaluated at $ (H_m)$ is  $ q_m |W_{G}
{H_m} | =0 $. This gives a contradiction. \end{proof}

   If  $G$ is a finite group and  $\phi$ is a
surjective map, then $ G $ is the trivial group.

\section{The mark  homomorphism   $\phi$ is a rational isomorphism}
  The topology on $ \Phi G$  can be  described in a way which
  makes it clear that the functions $ \phi ( [G/H]  )$, for $(H) \in \Phi G$,   generate
  $ C(G)\otimes_{\zz} \qq $ as a $\qq$-module.
\begin{lem} The topology on  $ \Phi G $ is the smallest topology
such that $$ V (K) = \{ (H) \in \Psi G \ | \ (H) \leq (K) \}
$$ is both an open and a closed subset of $\Phi G$ for all $
(K) \in \Phi G $.
\end{lem}
\begin{proof}  The definition of the Hausdorff topology shows that $ V (K) $
is closed. By Montgomery and Zippin's theorem   $ V (K) $  is also open for all $ (K) \in \Phi G $.

Since $\Phi G$ is a countable metric space it has a basis for
the topology consisting of open and closed sets. (For each element $x$ the function $d ( x , - ) $ is a continuous function to $\rr$.) Let $(K)$ be
in $ \Phi G$. Let $U $ be an open and closed neighborhood of
$(K)$ in $ \Phi G$. Since $V(K)$ is open and closed there is no loss of generality assuming  that $U$ lies inside $ V
(K)$. The set $ V(K) - U$ is open and closed. The collection
of sets $ V (H) $,  for all $(H)$  in $ V(K) - U$,  is an open
covering of the closed subspace $ V(K) - U$ of $\Phi G$.
 Since $\Phi G$ is a compact space
  there is a finite set $\{ ( H_1 )  , ( H_2 )  , \ldots , ( H_n )
 \}$ such that each $H_i $
is properly subconjugated to $K$ in $G$ and
$$  V(K) - U \subset \cup_{i=1}^{n} V (H_i )  .$$
Then $ V (K) - \cup_{i=1}^{n} V (H_i ) $ is an open and closed
neighborhood of $(K)$ contained in $U$. Hence $ V(K)$ and its complement $ \Phi  G - V(K)$, for $(K) \in \Phi G $, generates the topology on $\Phi G$.
\end{proof}
\begin{pro} \label{rational}
The map $$ \phi \otimes_{\zz} \qq  \colon  A(G) \otimes_{\zz} \qq \rarr
C(G) \otimes_{\zz} \qq $$ is an isomorphism.
\end{pro}
\begin{proof} It suffices to  prove that the characteristic function on each of
the open and closed subsets $ V (H)$ for $ (H) \in \Phi G$ are in the
rational image of $\phi$.
   Since $W_G H$  acts freely from the right on
 $  (G/H)^ K $ the values of $\phi ([G /H] ) (K) = | G/H^K | $ are
 multiples of $|W_G H |$.

Let $ V (H ,n )  $ be $ \{ (K) \ | \ | G/H^K | \geq n |W_G H
|\} $. Then  $ V ( H , 1) $ equals $ V (H )$. Define subsets
$$ U (H ,n ) = \{ (K) \ | \ | G/H^K | = n |W_G H | \} = V (H
,n ) - V(H , n+1 ) .$$ Since $\phi ([G /H ])$ is a continuous
function on a compact set, only finitely many of the $ U (H ,n
)$  are nonempty. By considering linear combinations of different  powers of
 $\phi ([G /H ])$  the characteristic function on  $ U (H ,n ) $
 is  in the rational image of $ \phi $, for each $n$.  Hence so is  the characteristic
 function on
 $V (H ) = \cup_{n \geq 1}  U (H ,n ) $.
\end{proof} 

\section{An alternative additive basis for $\mathbf{C(G)}$} Assume that $  H$ has  a finite Weyl group. Since
$W_G H$ acts freely on $ (G/H)^K$ it follows that  $\phi_{K} ([G/H]) $ is divisible by $ | W_G H|$ for all $(K)$.
 Hence the function
  $$ a_H = \frac{1}{  | W_G H |}   \phi ([G/H]) $$ is in $C(G)$. It is not possible to divide any further since  $ a_H (H) =1 $.
\begin{pro} \label{additivebasis}
The elements $ a_H$, for $(H) \in \Phi G$,  are linearly independent and generates
$C(G)$ as an abelian group.
\end{pro}
\begin{proof}
The elements $a_H$, for  $(H) \in \Phi G$,  are linearly independent in $C(G)$ by the proof of Lemma \ref{lem:phiinj}.

Any element $f$ in $C(G)$ is in the rational image of $\phi$ by Proposition \ref{rational}. It suffices to show that if
$$f =  \textstyle\sum_{ i=1}^{n} q_i \, a_{H_i}     $$ is in $C(G)$, where  $q_i \in \qq$, then    $ q_i \in \zz $, for all $i$.
Assume that $ (H_k )$ is maximal among the $(H_i )$ with $q_i \not
\in \zz $. Then $ f ( H_k ) = q_k + \sum_j q_j a_{H_j} (H_k )
$ where the sum is over subgroups $H_j$ that properly contains
a conjugate of $H_k$. This gives a contradiction, so all $ q_i
$ are integers.
\end{proof}

\section{Congruence relations describing the image $\mathbf{\phi ( A(G) )}$ in $\mathbf{C(G)}$}
 It is possible to describe the image $\phi (A(G))$ in
$C(G)$ by a set of congruence relations. There is one
congruence relation for each element in $\Phi G$.
\begin{lem}
\label{lem:congruences} Let $G$ be a finite group.
Then  $$
\textstyle\sum_{g \in G} \, |X^g | \equiv 0 \text{ mod }  |G| $$ for all finite $G$-sets $X$.
\end{lem}
\begin{proof}  Note that $\Sigma_{g \in G} \,|(G/H)^g|$  equals  $\Sigma_{k H} \, |k H k^{-1}| = |G|$ by rearranging the summation.
This implies that
   $ \Sigma_{g \in G} \,  |X^g | = |G|
|X /G| $. \end{proof}
\begin{pro}
\label{congruencerelations}
 Let $G$ be a compact Lie group.
An element $f \in C(G)$ is in the image of
 $\phi  \colon  A(G) \rarr
C(G)$ if and only if, for each  $(H) \in    \Phi G$, it satisfies the following congruence
relation $$ \textstyle\sum_{C} \, n_{C/H}
f (C) \equiv 0 \text{ mod } |W_G H| $$ where the sum is over
all $C$ such that  $H \lhd C$ and $ C/H$ is a cyclic group;   $n_{C/H}$
is the number of generators of the cyclic subgroup $C /H$.
\end{pro}
\begin{proof} For each $(H) \in \Phi G$
  apply Lemma \ref{lem:congruences}
  to $ N_G H / H$ acting on the finite $N_G
H/H$-sets $( G/L)^H$ for any subgroup $L$ of $G$ with finite Weyl group. The congruence
relation for $H$ is
$$ \textstyle\sum_{C} \, n_{C/H} \phi_C ( [ G/ L ]) \equiv 0 \text{ mod } |W_G H|
$$ for any $[G/ L ] $
where the sum is over all $H \lhd C$ such that $ C/H$ is a
cyclic group and $n_{C/H}$ is the number of different generators of the
cyclic subgroup $C /H$.
 Hence any element in the image of the
mark homomorphism satisfies these relations.

Any element in $C(G)$ can be written as a sum $\Sigma_K \,  m_K
a_K$ where  $m_K$ are integers
and all but finitely many of them
are zero. Assume   the element $\Sigma_K \,  m_K
a_K$ satisfies all the congruence
relations.  If  the integer   $m_K $ is
 divisible by $| W_G K|$ for all $K$, then  $\Sigma_K \,  m_K
a_K$ is in the image of $\phi$ since $m_K a_K = \frac{m_K}{  | W_G K |}   \phi ([G/K])$. Assume that this is not the case and
let $(H)$ be  maximal   for which $m_H$ is not divisible
by $| W_G H|$. Since  $m_K a_K$ satisfies all the congruence relations for all $K$ properly containing a  conjugate of $H$, the   congruence relation for $(H)$
gives that $m_H
\equiv 0 \text{ mod } |W_G H|$.
This is  a contradiction. Hence
$|W_G K|$ must divide $m_K$ for all $K$.  \end{proof}

It is a theorem of tom Dieck that there is a greatest common
divisor of  $|W_G H |$ for all subgroups  $H$ of $G$ \cite{tdf}.
 Let $|G|$ denote the greatest common divisor of $|W_G H|$ for all subgroups $H$ of $G$ with finite Weyl group.  If $G$ is a finite group, then  $|G|$ equals the number of elements
in  $G$. The congruence relations gives the following.
\begin{cor} There is an inclusion  $|G| C(G) \subset \phi (A(G))$.
\end{cor}
For certain purposes the number $|G|$ serves as the order of the compact Lie group.  A generalization of the  Artin induction theorem to  compact Lie groups make use of  a whole family of different orders of compact Lie groups \cite{abi}.  These orders all reduce to the number of elements in $G$ when $G$ is a finite group.

\section{The prime ideal spectrum  of $\mathbf{A(G)}$ }
The ring map $ \phi  \colon  A(G) \rarr C(G) $ is an integral
extension since $C(G)$ is additively generated by idempotent
elements.
  Hence by the going up theorem in commutative algebra the map
 $$ \spec \phi \colon  \spec C(G) \rarr \spec A(G)$$ is surjective. That is,
   all the prime ideals of $ A(G) $
are  of the form $ \phi^{ -1 } ( P ) $ for some prime ideal
$P$ in $ C(G)$.

The  prime ideals in $C(G)$ are all obtained by  applying  the
following
 standard lemma to $ X = \Phi G$ and $R = \zz$.
\begin{lem} \label{space}  Let $ X $ be a totally disconnected compact Hausdorff space,
and let $R$ be a ring. Then there is an homeomorphism
$$ q \colon   X \times \spec R \rarr   \spec  C ( X ,  R ) $$
given by sending $ (x , P) $ to the prime ideal $$ \{ f \in C(X
, R ) \ | \ f ( x) \in P \} . $$
\end{lem}
\begin{proof} (Outline of a proof.)  Let $c $ be an ideal in
$C ( X , R )$. Let $$ I (c ,x ) = \{ f (x) \ | \ f \in c \} $$
be the ideal of all the values of the functions in $c$  at
$x$. If $J$ is a prime ideal, then there is a unique $x_J$
such that $ 1 \not\in I (c , x_J )$. The map sending $J$ to
$ ( x_J , I (c , x_J ) ) $ is an inverse to $q$. The two maps
are continuous. \end{proof}

 The prime ideals of $ A (G) $ are all of the form
$$ \begin{aligned} q ( H , 0)  & = \{  x \in A(G)  \ | \ \phi_H  (x )   =0  \, \}  \\
 q ( H , p) & = \{ x \in A(G)  \ | \ \phi_H (x)   \in p \zz \} \end{aligned} $$
for $p$ any prime number and $(H) $ in $ \Phi G $. This follows by pulling back the prime ideals of $C(G)$   along $\phi$.

\begin{lem} The Burnside ring $A(G)$ has krull dimension 1 for all compact Lie groups $G$. The maximal ideals of $A(G)$ are $q(H ,p)$ for $ (H) \in \Phi G $ and primes  $p$. The minimal prime ideals of $A(G)$ are  $q(H , 0)$  for $ (H) \in \Phi G $.
There is an inclusion $ q (H, 0 ) \subset q (K ,p)$ if and only if $ q (H ,p) = q (K ,p)$.  \end{lem}
\begin{proof}
 The quotient rings of $A(G)$  are  $ A(G) / q ( H , p ) \cong \zz / p$ and $A(G) / q (H , 0) \cong \zz$. The isomorphisms are induced by $\phi_H$.  The ideal $q(H,p)$ is maximal since $ A(G) / q(H,p)$ is a field. There are no  proper containments among the prime ideas of the form $q (H, 0)$ because  a quotient of $\zz$ by a nonzero ideal is finite. Hence $q(H,0)$ is a  minimal prime ideal.

 There is an inclusion $ q (H, 0 ) \subset q (H ,p)$ for  $(H) \in \Phi G$ and $p$ a prime number. Hence the equality $ q (H ,p) = q (K ,p)$ implies that  $ q (H, 0 ) \subset q (K ,p)$.  Conversely, assume that $ q (H, p ) $ and  $ q (K ,p)$ are not equal. Since they are maximal ideals,  there is then an element $x \in A(G)$ such that $ x \in q(H,p)$ and $x \not\in q ( K,p)$.  Define the  element $y = x - \phi_H (x) \cdot 1 $.  Then $y \in   q(H,0)$ and $y \not\in q ( K,p)$. So $ q (H, 0 ) $ is not a subset of $ q (K ,p)$.
\end{proof}

If $ K \triangleleft H $ is a normal subgroup with
a $p$-group quotient, then   $q (K,p) = q(H,p) $. (This follows since for any finite   $\zz/p$-set $X$, the difference  $ |X| -|X^{\zz /p } |$ is divisible by $p$.)
Given a closed subgroup $H$ of $G$. Let  $ H'_p$ be the minimal normal  subgroup  of  $H$ such that  $H / H'_p$ is a  finite $p$-group.   The  group $ H_{p}'$ might not have finite Weyl group as a
subgroup of $G$. The  conjugacy class $(H_p) =
\omega ( H'_p )$ is called the $p$-perfection of $H$ in $G$ (See section \ref{sect:mark}).

\begin{lem} \label{maxideals} Choose a representative $H_{p}$  for the conjugacy class $(H_{p})$. There is an extension $H^p$ of $ H_{p}$ such that $ H^p / H_p$ is a finite $p$-group and $| W_G H^p|$ is relative prime to $p$. The conjugacy class $(H^p)$ only depends on $(H)$.  There are identities  $$ (( H_p )^p) = (H^p ) \ \text{and} \ (( H^p
)_p) = (H_p ) ,$$
and
$$ q (H,p) = q (H^p ,p) = q (H_p ,p),  $$ for    $(H) \in \Phi G$ and prime numbers $p$.
\end{lem}
\begin{proof}
 Let $H^p$ be the preimage of a  Sylow $p$-subgroup of $ N_G H_p / H_p$ in $N_G H_p$ (unique up to conjugation).  Assume  the order of the Weyl group $ W_G H^p$ is  not  relative prime to $p$. Then  there are  two proper normal extensions
 $$ H_p \lhd H^p \lhd K$$
 both of whose quotients are $p$-groups. Then   $ H_p $ must be normal in $ K$ since $p$-groups are solvable.  This gives a  contradiction since  $H^p /  H_{p}$ was a maximal $p$-group in $N_H H_p / H_{p}$.  The equality $( H_p )^p = (H^p )$ follows by the definition of $H^p$.  There is an inclusion $ (H^p)'_p \leq H_p$.
 The normal subgroup $ (H^p)'_p \cap H'_p $ of $H$ must equal $H'_p$ since   $$ H'_p / (H^p)'_p \cap H'_p \cong H_p /  (H^p)'_p    $$ is a $p$-group.  Hence $ (H^p)'_p $ equals $  H_p$, and so $ (H^p)_p = H_p$.
\end{proof}

\begin{exmp} Consider the nontrivial semidirect product  $S^1 \rtimes \zz / 2$, and $p=2$.   The subgroup  $ H = 0 \rtimes \zz /2$ has finite Weyl group.  In this case  $ H_2 = S^1 \rtimes 0 $ and $ H^2 = S^1 \rtimes \zz / 2$. \end{exmp}

\begin{lem}  There is an equality
 $  q ( H ,0 ) =  q ( K ,0)$ if and only if  $   (H )= (K )$.
 There is an equality
 $  q ( H ,p ) =  q ( K ,q )$ if and only if $p =q$ and  $   (H^p )= (K^p )$.
\end{lem}
\begin{proof}
Note that $ \phi ( [G/ H ]) (H) =| W_G H| $ and  $ \phi ( [G/ H ]) (K) =0 $, whenever $ K$ is not subconjugated to
$H$. A  closed subgroup of $G$ can not  be properly subconjugated to itself. Hence if $ (H) \not=(K) $ in $ \Phi G$, then
$q (H,0) \not= q(K,0)$.

By  Lemma \ref{maxideals} there are identifications  $ q (H, p ) =  q (H^p , p ) $ and  $ q (K ,p )=  q (K^p , p ) $.  Since $A(G) / q (H , p ) \cong \zz /p$, the equality  $  q ( H ,p ) =  q ( K ,q )$ implies that $p =q$.
Both $ (G/ H^p )^{K^p}$ and $ (G/ K^p )^{H^p}$ are nonempty since the Weyl groups of $H^p$ and $K^p$ are not  divisible by  $p$. Hence $(H^p)$ and $(K^p)$ must be equal.
\end{proof}

\begin{lem}  \label{omeg} Let  $ H \leq J \leq K $ be subgroups of $G$. If  $\omega (H ) $ is equal to $
\omega (K) $, then $\omega (H ) = \omega (J) = \omega (K) $.
\end{lem}
\begin{proof}  The conjugacy class  $ \omega (H)$ is  the conjugacy class of
$ H T_H $ where $T_H$ is a maximal torus in $ C_G H$.
Similarly for $J$ and $K$. Since $ C_G (K) \leq C_G ( J) \leq C_G H$ the maximal tori can be chosen such that  $ T_K \leq T_J \leq
T_H $.
The torus  $ T_H$ is a
maximal torus in $ C_G (H T_H)$ and  $ T_K$ is a maximal
torus in $ C_G (H T_K)$. The
   assumption that $  H T_H $ and $ K T_K $ are conjugate subgroups in $G$ gives that that $T_H$ is subconjugated to $T_K$, hence  $T_H = T_K$. Since
 $ T_K \leq  T_J \leq T_H $ this implies that  $ T_H = T_J = T_K $. The result now follows since
 $ H T_H = J  T_J = K  T_K $.
\end{proof}

The next  result was first proved   by  Bauer and May  \cite{max}.
\begin{pro} \label{bm} Let $ H \leq J \leq  K $ be subgroups of $G$. If   $  q ( H ,p ) $ is equal to $ q ( K ,p )$, then   $  q ( H ,p ) =  q ( J ,p ) =  q ( K ,p )$.
\end{pro}
\begin{proof}
Assume that $H$ is a subgroup of $J$.  Let $J'_p$ be the smallest normal subgroup of $J$ such that $ J / J'_p$ is a $p$-group.
Then $J'_p \cap H$ is a normal subgroup of $H$ such that $ H / J'_p \cap H \leq J / J'_p$ is a $p$-group.  Hence $ H'_p \leq J'_p \leq  K'_p $ and  Lemma \ref{omeg} gives the result.
\end{proof}

The space  $\Phi G$ has only finitely many elements if and
only if the Weyl group of a maximal torus $T$ of $G$ acts
trivially on $T$ \cite[5.10.8]{td}. Hence   $A(G)$
is a Noetherian ring if and only if the Weyl group of a
maximal torus $T$ of $G$ acts trivially on $T$.

\section{Idempotent and unit  elements in $\mathbf{A(G)}$}

The idempotent elements in $ A(G)$, and the idempotent elements in
 $A(G)$
with  a set of primes inverted,  have been completely
described   \cite{tdi} \cite{dre} \cite{fol}  \cite{yos}.
For example,  when   $G$ is a finite group, then   $A(G)$ has no idempotent elements
different from  0 and 1 if and only if $G$ is a
solvable group
\cite[5.11.4]{td}.
 This fact
was emphasized in \cite{dre}.

Let $\pi$ be a collection of  prime numbers. A  group $H$ is $\pi$-perfect if the group of components of $H$ does not have a nontrivial solvable quotient $\pi$-group.
 Let $\Phi_{\pi} (G)$ be the subspace of $\Phi G$ consisting of conjugacy classes of $\pi$-perfect subgroups of $G$ with finite Weyl group.

 Let $X$ be  a topological space, and let $\Pi_0 (X)$   denote the space of components (with the quotient topology from $X$). Let  $R$ be a ring, and  let
  $R_{(\pi)}$ denote the localization of $R$ obtained by inverting all primes not in the set $\pi$.
 There is a map $ \beta \colon \Phi_{\pi} (G) \to \Pi_0 (\spec A(G)_{(\pi)}) $ defined by sending $(H)$ to the component of the prime ideal $q(H, 0)$.
 The following was proved  in \cite[Proposition 3.3]{fol}.
\begin{pro}
The map  $$ \beta \colon \Phi_{\pi} (G) \to \Pi_0 (\spec A(G)_{(\pi)}) $$ is a homeomorphism.
\end{pro}

There is a close connection between idempotent elements in
$A(G)$ and unit elements in $A(G)$.  The following  is a consequence  of
the embedding of $A(G)$ into $C(G)$. If $e$ is an idempotent element in $A(G)$, then
$2e - 1$ is a unit element in $A(G)$. If $u $ is a unit element in
$A(G)$, then $ \frac{ \phi (u) +1 }{2}$ is an idempotent element in $
C(G)$.   If $G$ is a compact Lie group and $|G|$ is odd, then $ \frac{  u +1 }{2}$ satisfies
the congruence relations of   Proposition \ref{congruencerelations}, because both $u$ and $1$
satisfy the relations and $|W_G H|$ is not divisible by $2$ for any $(H) \in \Phi G$.  Proposition \ref{congruencerelations}
gives that $ \frac{  u +1 }{2}$ is in $A(G)$. Hence there is a bijection between idempotent elements and unit elements in $A( G)$ when $|G|$ is  odd.

There is a homomorphism $ R (G ; \rr ) \rarr A(G)^{\times} $
given by sending a real representation $ V $  to the function $
(-1)^{\text{dim} V^H} $ \cite[5.5.9]{td}.
 Tornehave has proved that this  map is surjective
 when $G$ is  a finite 2-group \cite{tor}.

 \section{A map from the Burnside ring to the  representation ring} Let $k$ be a field.
There is a canonical map $$ A(G) \rarr R(G; k) $$ from the
Burnside ring of $G$ to the representation ring of $G$ with
coefficients in  $k$. The map is given by sending $ G
/H$ to the alternating sum of the $G $-representations $ H^i (
G /H ; k )$ \cite{bside1}.
Here $ H^i (
G /H ; k )$ is the $i$-th singular cohomology of $G/H$ with coefficients in $k$ endowed  with a left $G$-action induced by the left $G$-action on $G/H$.
The  map $ A(G) \rarr R(G; k)$ is neither injective
nor surjective in general. However if $P$ is a finite
$p$-group, then  $ A(P)  \rarr R(P ; \qq
) $ is surjective \cite{per}.

 The composition $ A(G ) \rarr R( G ; \rr ) \rarr A(G)^{\times} $ is called the exponential map of the Burnside ring. This map  is surjective if $ G$ is a $2$-group with  no subquotients isomorphic to the dihedral group of
order 16 \cite{erg}.

\section{Modules over $A(G)$}
Modules over $A(G)$ have been studied by tom Dieck and Petrie \cite{mod}.
Much attention has been given to invertible modules over $A(G)$ \cite{ pic} \cite{picLie}  \cite{picmod}.
These are closely related to stable homotopy representations.
A homotopy representation is a retract of a finite $G$-CW complex $X$ such that $X^H$ is homotopy equivalent to $S^{n(H)}$ for some integer $n(H)$, for each subgroup $H$ in $G$.
A stable homotopy representation is a suspension spectrum of a homotopy representation.
 Stable homotopy representation are
 exactly the    invertible
objects  in the stable equivariant homotopy category \cite{picLie} \cite{flm}.

The finite groups $G$ such that there are only a finite number of  finitely generated indecomposable $A(G)$-modules (which are free
over the integers) are characterized by Reichenbach  in  \cite{rei}.

\section{The Burnside ring in equivariant stable homotopy theory}
\label{sec:htptheory}
Let $X$ be a finite $G$-CW-complex.
 The (categorial) Euler characteristic of $\Sigma^{\infty}_G X$ turns out to be the
stable homotopy class
$$ \Sigma^{\infty}_G S^0 \stackrel{\tau}{\rarr} \Sigma^{\infty}_G X_+
\stackrel{c}{\rarr} \Sigma^{\infty}_G S^0 $$ where the first
map is the transfer map and the last map is the collapse map
that sends $ X$ to a point \cite[Chapter XVII]{ala}. This defines a homomorphism
$$ \chi  \colon  A (G) \rarr \pi_{0}^G ( \Sigma^{\infty}_G  S^0 ) =  \pi_{0}^G
( S^0 ) . $$ There is a degree homomorphism  $$ d  \colon  \pi_{0}^G ( S^0 )
\rarr C (G) $$ which  sends a stable self map $ h  \colon
\Sigma^{\infty}_G S^0 \rarr \Sigma^{\infty}_G S^0 $ to the
continuous function that map
 $(K)$ to the degree of (the geometric fixed points)  $ h^K $ as a stable self map of
 $ \Sigma^{\infty} S^0 $.
The composite  $ d \circ \chi $ equals $ \phi $.

\begin{thm}  \label{htp} The Euler characteristic map
$$ \chi  \colon  A (G) \rarr \pi_{0}^G (  S^0 )   $$ is an isomorphism.
\end{thm}
The injectivity of $\chi$ follows from the injectivity of $\phi$. The
surjectivity of $\chi$ requires a more careful   understanding of
$\pi_{0}^G (S^0 )$.
 It suffices  to show that $ \pi_{0}^G ( S^0 ) $ is additively
generated by maps of the form $ \Sigma^{\infty}_G S^0
\stackrel{\eta}{\rarr} \Sigma^{\infty}_G X_+
\stackrel{c}{\rarr} \Sigma^{\infty}_G S^0 $.
This follows from the spectrum level Segal--tom Dieck splitting
theorem \cite[IV.9.3]{lms}.

Another  proof  consists of using equivariant obstruction theory to
show that $d$ is injective and to show that every function in the
image of $d$ satisfies the congruence relations of Proposition
\ref{congruencerelations}. Since $d$ is injective
the map  $\chi$ must then  be surjective. The details are given in
\cite[Chapter 8]{td}.

A variation of Theorem \ref{htp}  is the    Segal conjecture, proved by Carlsson \cite{car}.
 Let $G$ be  a finite group and let $EG$ be  a free contractible $G$-space.  The stable $G$-homotopy classes of
  maps from $\Sigma^{\infty}_G EG_+ $ to $ \Sigma^{\infty}_G S^0$ is isomorphic to the  completion of the Burnside ring  at its augmentation ideal.
This is a deep result of importance in homotopy theory.

\end{document}